\documentclass[12pt]{amsart}
\usepackage{amsfonts,amssymb}
\usepackage[all]{xypic}
\usepackage{amsmath}
\usepackage{amsthm,amscd,latexsym}
\usepackage{color}

\newtheorem{theorem}{Theorem}[section]
\newtheorem{corollary}[theorem]{Corollary}
\newtheorem{Definition}[theorem]{Definition}
\newtheorem{lemma}[theorem]{Lemma}
\newtheorem{proposition}[theorem]{Proposition}
\newtheorem{Example}[theorem]{Example}
\newtheorem{Remark}[theorem]{Remark}

\newenvironment{remark}{\begin{Remark}\begin{em}}{\end{em}\end{Remark}}
\newenvironment{example}{\begin{Example}\begin{em}}{\end{em}\end{Example}}
\newenvironment{definition}{\begin{Definition}\begin{em}}{\end{em}\end{Definition}}

\newcommand{\A}{\mathcal{A}}
\newcommand{\R}{{\mathbb R}}
\newcommand{\bP}{{\mathbb P}}

\newcommand{\RP}{{\mathbb R}^+}

\newcommand{\N}{\mathbb N}

\newcommand{\ve}{\varepsilon}

\DeclareMathOperator{\ua}{\uparrow\!}
\DeclareMathOperator{\da}{\downarrow\!}
\DeclareMathOperator{\Pro}{\mathcal{P}}
\DeclareMathOperator{\argmin}{\mathrm{arg\, min}}

\begin{document}
\title[The Stochastic order of probability measures]
{The Stochastic order of probability measures on ordered metric spaces}

\author{Fumio Hiai, Jimmie Lawson, Yongdo Lim}

\address{Tohoku University (Emeritus), Hakusan 3-8-16-303, Abiko 270-1154, Japan}\email{hiai.fumio@gmail.com}
\address{Department of Mathematics, Louisiana State University,
Baton Rouge, LA70803, USA}\email{lawson@math.lsu.edu}
\address{Department of Mathematics, Sungkyunkwan University, Suwon 440-746, Korea} \email{ylim@skku.edu}
\maketitle

\begin{abstract}
The general notion of a stochastic ordering is that one probability distribution is smaller than a second one if the second attaches more
probability to higher values than the first.  Motivated by recent work on barycentric maps on spaces of probability measures on ordered
Banach spaces,  we introduce and study a stochastic order on the space of probability measures $\mathcal{P}(X)$, where
$X$ is a metric space equipped with a closed partial order, and derive several useful equivalent versions of the definition. We establish the antisymmetry
and closedness of the stochastic order (and hence that it is a closed partial  order) for the case of a partial order on a Banach space induced
by a closed normal cone with interior.  We also consider order-completeness of the stochastic order for a cone of a finite-dimensional Banach space and  derive a version of the
arithmetic-geometric-harmonic mean inequalities in the setting of the associated probability space on positive matrices.
\end{abstract}



\allowdisplaybreaks

\section{Introduction}   The stochastic order for random variables $X,Y$ from a
probability measure space $(M,P)$ to $\R$ is defined by $X\leq Y$ if $P(X>t)\leq P(Y>t)$ for all $t\in \R$.  This notion
extends directly to random variables into $\R^n$ equipped with the coordinatewise order.  Alternatively one can define
a stochastic order on the Borel probability measures on $\R$ or $\R^n$ by $\mu\leq \nu$ if for each $s\in\R$,
$\mu(s<t)\leq \nu(s<t)$, where $(s<t):=\{t\in\R: s<t\}$.  One then has for random variables $X,Y$, $X\leq Y$ in the
stochastic order if and only if $P_X\leq P_Y$, where $P_X,P_Y$ are the push-forward probability measures with respect to
$X,Y$ respectively.

There are important metric spaces which are equipped with a naturally defined partial order, for example the open cone $\mathbb{P}_n$ of positive definite
matrices of some fixed dimension, where the order is the Loewner order.   One can  use the Loewner order to define an order on
$\mathcal{P}(\mathbb{P}_n)$, the space of Borel probability measures, an order that we call that stochastic order, as it generalizes the case of $\mathbb{R}$
or $\R^n$.

In this paper we broadly generalize the stochastic order  to an order on the set of Borel probability measures on a partially ordered metric space.
We develop basic properties of this order and specialize to the setting of normal cones in Banach spaces to show that the stochastic order
in that setting  is indeed a partial order.

In Section 3 we give the general definition of the stochastic order on $\mathcal{P}(X)$ for a partially ordered metric space $X$ and derive several useful
alternative formulations.  In Section 4 we show for normal cones with interior that the stochastic order on $\mathcal{P}(X)$ is indeed a partial order (the
antisymmetry being the nontrivial property to establish).  In Section 5 we show in the normal cone setting that the stochastic partial order is a
closed order with respect to the weak topology, and hence with respect to the Wasserstein topology.  In Section 6 we consider the order-completeness
of $\mathcal{P}(X)$, and in Section 7 derive a version of the arithmetic-geometric-harmonic means inequality in the setting of the probability space
$\mathcal{P}(\mathbb{P})$ on the cone $\bP$ of positive invertible operators on a
Hilbert space.

 In what follows $\RP=[0,\infty)$.

\section{Borel measures}
In this section we recall some basic results about Borel measures on metric spaces that will be needed in what follows.  As usual the Borel
algebra on a metric space $(X,d)$ is the smallest $\sigma$-algebra containing the open sets
and a finite positive Borel measure is a countably additive measure $\mu$ defined on the Borel sets such that $\mu(X)<\infty$.  We work exclusively with finite positive
Borel measures, primarily those that are probability measures.

Recall that a Borel measure $\mu$ is $\tau$-\emph{additive} if
$\tau(U)=\sup_\alpha \tau(U_\alpha)$ for any \emph{directed} union
$U=\bigcup_\alpha U_\alpha$ of open sets.  The measure $\mu$ is said
to be \emph{inner regular} or \emph{tight} if for any Borel set $A$
and $\ve>0$ there exists a compact set $K\subseteq A$ such that
 $\mu(A)-\ve<\mu(K)$. A
tight finite Borel measure is also called a Radon measure.

Probability on metric spaces has been carried out primarily for separable metric spaces, although results exist  for the non-separable setting.
We recall the following result, which can be more-or-less cobbled together from results in the literature; see \cite{La17} for more details.
\begin{proposition}\label{P:La}
A finite Borel measure $\mu$ on a metric space $(X,d)$ has separable
support. The following three conditions are equivalent$:$
\begin{itemize}
\item[(1)] The support of $\mu$ has measure $\mu(X)$.
\item[(2)] The measure $\mu$ is $\tau$-additive.
\item[(3)] The measure $\mu$ is the weak limit of a sequence of finitely supported measures.
\end{itemize}
If in addition $X$ is complete, these are also equivalent to:
\begin{itemize}
\item[(4)] The measure $\mu$ is inner regular.
\end{itemize}
\end{proposition}

\begin{proof} For a proof of separability and the equivalence of the first three conditions, see \cite{La17}.  Suppose (1)--(3) hold and
$X$ is complete.  Let $\mu$ be a finite Borel measure.
Then the support $S$ of $\mu$ is closed, separable and has measure
$1$. Let $A$ be any Borel measurable set.  Then $\mu(A\cap
(X\setminus S))=0$ since $\mu(X\setminus S)=0$, so $\mu(A)=\mu(A\cap
S)$.  Since the metric space $S$ is a separable complete metric
space, it is a standard result that $\mu\vert_S$ is an inner regular
measure.  Thus for $\ve>0$ there exists a compact set $K\subseteq
S\cap A\subseteq A$ such that $\mu(A)=\mu(A\cap S)<\mu(K)+\ve$.

Conversely suppose $\mu$ is inner regular.  If $\mu(S)<\mu(X)$ for the support $S$ of $\mu$, then for $U=X\setminus S$, $\mu(U)>0$.
By inner regularity there exists a compact set $K\subseteq U$ such that $\mu(K)>0$.  Since $K$ misses the support of $\mu$, for each
$x\in K$, there exists an open set $U_x$ containing $x$ such that $\mu(U_x)=0$.  Finitely many of the $\{U_x\}$ cover $K$, the finite union has
measure $0$, so the subset $K$ has measure $0$, a contradiction.  So the support of $\mu$ has measure $\mu(X)$.
\end{proof}

\begin{remark} Finite Borel measures on separable metric spaces are easily shown to be $\tau$-additive and hence satisfy the other equivalent conditions of Proposition \ref{P:La}.
Finite Borel measures that fail to satisfy the previous conclusions are rare. Indeed it is a theorem that in a complete metric space $X$ there exists
a finite Borel measure that  fails to be inner regular if and only if the minimal cardinality $w(X)$ for a basis of open sets of $X$ is a measurable
cardinal; see volume 4, page 244 of \cite{Fr}. The existence of measurable cardinals is an axiom independent of the basic Zermelo-Fraenkel axioms of set theory
and thus if its negation is assumed, all finite Borel measures on complete metric spaces satisfy the four conditions of Proposition \ref{P:La}.
\end{remark}

\section{The stochastic order}
\emph{We henceforth restrict our attention to the set of Borel
probability measures on a metric space $X$ satisfying the four
conditions of Proposition \ref{P:La} and denote this set} $\Pro(X)$.
For complete separable metric spaces  the set $\Pro(X)$  consists of
all Borel probability measures, which are automatically $\tau$-additive in this case.

\begin{definition} A \emph{partially ordered topological space} is a space equipped with a closed partial order $\leq$, one for which $\{(x,y):x\leq y\}$ is closed
in $X\times X$.  
\end{definition}
For a nonempty subset $A$ of a partially ordered set $P$, let $\ua A:=\{y\in P:\exists x\in A,\, x\leq y\}$.  The set $\da A$ is defined in an order-dual fashion.
A set $A$ is an \emph{upper set} if $\ua A=A$ and a \emph{lower set} if $\da A=A$.  We abbreviate $\ua\{x\}$ by $\ua x$ and $\da\{x\}$ by $\da x$.

\begin{lemma}\label{L:SO1} A partially ordered topological space is Hausdorff. If $K$ is a nonempty compact subset, then $\ua K$ and $\da K$ are closed.
\end{lemma}

\begin{proof} See Section VI-1 of \cite{GS}. \end{proof}

 The following definition captures in the setting of ordered topological spaces the notion that higher values should have higher probability.
\begin{definition} For a topological  space $X$ equipped with a closed partial order, the \emph{stochastic order} on $\Pro(X)$ is defined by
$\mu\leq \nu$ if $\mu(U)\leq \nu(U)$ for each open upper set $U$.
\end{definition}

\begin{proposition}\label{P:SO2}
Let $X$ be a metric space equipped with a closed partial order. Then
the following are equivalent for $\mu,\nu\in\Pro(X):$
\begin{itemize}
\item[(1)]  $\mu\leq \nu;$
\item[(2)] $\mu(A)\leq \nu(A)$ for each closed upper set $A;$
\item[(3)]  $\mu(B)\leq \nu(B)$ for each upper Borel set $B$.
\end{itemize}
\end{proposition}

\begin{proof} Clearly (3) implies both (1) and (2).

(1)$\Rightarrow$(3):  Let $B=\ua B$ be a Borel set.  The $A=X\setminus B$ is also a Borel set.  Let $\ve>0$.  By inner regularity there exists
a compact set $K\subseteq A$ such that $\nu(K)> \nu(A)-\ve$.  By Lemma \ref{L:SO1} $\da K$ is closed, and $K\subseteq \da K\subseteq \da A=A$.  Thus $\nu(\da K)>\nu(A)-\ve$.  The complement $U$ of $\da K$ is an open upper set.
Taking complements we obtain
\begin{align*}
\mu(B)&=1-\mu(A) \leq 1-\mu(\da K)=\mu(U) \leq \nu(U) =1-\nu(\da K) \\
&<1-\nu(A)+\ve=\nu(B)+\ve.
\end{align*}
Since $\mu(B)<\nu(B)+\ve$ for all $\ve>0$, we conclude $\mu(B)\leq \nu(B)$.

(2)$\Rightarrow$ (3): We can approximate any Borel upper set $B$ arbitrarily closely from the inside with compact subsets $K$ and their upper sets $\ua K$ will
be closed sets that are at least as good approximations.  The Borel measure $\nu$ dominates $\mu$ on these  closed upper sets and hence also
in the limiting case of $B$.
\end{proof}

\begin{remark}\label{R:SO3}
By taking complements one determines that each of the preceding equivalences has an equivalent version for lower sets with the
inequalities in (2) and (3) reversed.
\end{remark}

 We turn now to functional characterizations of the stochastic order on $\Pro(X)$ for $X$ a metric space equipped with
 a closed partial order.
 In the next proposition, we write $\int_X f(x)\, d\mu(x)$ or simply $\int_X f\,d\mu$ for any Borel function $f:X\to \R^+$ and $\mu\in\Pro(X)$,
 where the integral is possibly infinite.
We say that $f$ is \emph{monotone} if $x\leq y$ in $X$ implies $f(x) \leq f(y)$.

\begin{proposition}\label{P:SO3}
Let $X$ be a metric space equipped with a closed partial order. Then
the following are equivalent for $\mu,\nu\in\Pro(X):$
\begin{itemize}
\item[(1)] $\mu\leq \nu;$
\item[(2)] for every monotone  $($bounded$)$ Borel function $f:X\to \RP$, $\int_X f\,d\mu\leq \int_X
f\,d\nu;$
\item[(3)] for every  monotone $($bounded$)$ lower semicontinuous $f:X\to \RP$, $\int_X f\,d\mu\leq \int_X f\,d\nu$.
\end{itemize}
\end{proposition}

\begin{proof} The implications that the general case implies the bounded case in items (2) and (3) are trivial.

(1)$\Rightarrow$(2): Assume $\mu\leq \nu$ and let $f$ be a non-negative monotone Borel measurable
function on $X$.   For each $n$, define $\delta_n:\RP\to\RP$ by $\delta_n(0)=0$,
$\delta_n(t)=(i-1)/2^n$ if $(i-1)/2^n<t\leq i/2^n$ for some
integer $i$, $1\leq i\leq n2^n$, and $\delta_n(t)=n$ for $n<t$.
Note that the ascending step function $\delta_n$
has finite image contained in
$\R^+$ and that the sequence $\delta_n$ monotonically increases
to the identity map on $\RP$. Hence $f_n:=\delta_nf$, the composition of $\delta_n$ and $f$,
monotonically increases to $f$. One verifies directly that the
step function  $f_n$ has an alternative description given by
\[ f_n =\sum_{i=1}^{n2^n} \frac{1}{2^n}\chi_{f^{-1}(]i/2^n,\infty))},\]
where $\chi_A$ is the characteristic function of $A$.  Since the sequence $\{f_n\}$
converges pointwise and monotonically to $f$, we conclude that $\int_X f\, d\mu=
\lim_n\int_X f_n\, d\mu$, and similarly for $\nu$.  Since  $f^{-1}(]i/2^n,\infty))$ is an upper
Borel set,  by Proposition \ref{P:SO2} $\mu(f^{-1}(]i/2^n,\infty)))\leq \nu(f^{-1}(]i/2^n,\infty)))$
for each $i$, so $\int_X f_n\, d\mu\leq\int_X f_n\,d\nu$ for each $n$, and thus in the limit
$\int_X f\,d\mu\leq \int_X f\,d\nu$.

(2)$\Rightarrow$(3): Since a lower semicontinuous function is a Borel measurable function, (3) follows immediately from (2).

(3)$\Rightarrow$(1):  The characteristic function $\chi_U$  is bounded, lower semicontinuous,  and monotone for $U$ an open upper set
and hence $\mu(U)=\int_X \chi_U\,d\mu\leq \int_X \chi_U\ d\nu=\nu(U)$.

\end{proof}

Call a real function $f$ on a partially ordered set $X$ \emph{antitone} if it is order
reversing, i.e., $x\leq y$ implies $f(x)\geq f(y)$.

\begin{corollary}\label{C:SO4}
Let $X$ be a metric space equipped with a closed partial order. Then
the following are equivalent for $\mu,\nu\in\Pro(X):$
\begin{itemize}
\item[(1)]  $\nu\leq \mu;$
\item[(2)] for every antitone  $($bounded$)$ Borel function $f:X\to \RP$, $\int_X f\,d\mu\leq \int_X
f\,d\nu;$
\item[(3)] for every  antitone $($bounded$)$ lower semicontinuous $f:X\to \RP$, $\int_X f\,d\mu\leq \int_X f\,d\nu$.
\end{itemize}
\end{corollary}

\begin{proof} Every partially ordered set has a dual order, namely the converse $\geq$ of $\leq$ is taken for the partial order.
Let $X^{od}$ denote the order dual of $X$.  Note that a subset $A$  of $X$ is an upper set in $(X,\leq)$ if and only if it is a lower
set in $X^{od}$.  Using Remark \ref{R:SO3}, one sees that $\mu\leq \nu $ with respect to  $(X,\leq)$ if and only if
$\nu\leq \mu$ with respect to  $X^{od}$.  Since antitone functions convert to monotone functions in the  order dual  of $X$, the corollary
follows from applying the previous proposition to the order dual.
\end{proof}

Finally we consider sufficient conditions for one to define the stochastic order in terms of continuous monotone functions.
\begin{proposition}\label{P:SO5}
Suppose that $(X,d)$ is a metric space equipped with a closed
partial order satisfying the property that given $x\leq y$ and
$x_1\in X$, there exists $y_1\geq x_1$ such that $d(y,y_1)\leq
d(x,x_1)$.  Then for $\mu,\nu\in\Pro(X)$ the following are
equivalent$:$
\begin{itemize}
\item[(1)]  $\mu\leq \nu;$
\item[(2)]  For every continuous $($bounded$)$ monotone $f:X\to\R^+$, $\int_X f\,d\mu\leq \int_X f\,d\nu;$
\item[(3)]   For every continuous $($bounded$)$ antitone $f:X\to\R^+$, $\int_X f\,d\nu\leq \int_X f\,d\mu$.
\end{itemize}
\end{proposition}

\begin{proof}
That (1) implies (2) follows from Proposition \ref{P:SO3} and (1)
implies (3) by Corollary \ref{C:SO4}.

(3)$\Rightarrow$(1): Let $V$ be an open lower set with complement $A$, a closed
upper set. For each $n\in\N$, define $f_n:X\to[0,1]$ by $f_n(x)=\min\{nd(x,A),1\}$
and note that $f_n$ is a continuous function into $[0,1]$. To show $f_n$ is antitone,  we
note for any $x\leq y$  and $x_1\in A$, there exists $y_1\geq x_1$
such that $d(y,y_1)\leq d(x,x_1)$.  It follows from $x_1\leq y_1$
that $y_1\in A$, hence $d(y,A)\leq d(x,x_1)$, and thus $d(y,A)\leq
d(x,A)$ since $x_1$ was an arbitrary point of $A$.  Hence
$$ f_n(y)=\min\{nd(y,A),1\}\leq\min\{nd(x,A),1\}=f_n(x).$$
It follows directly from the definition of $f_n$ that the sequence $\{f_n\}$ is an monotonically increasing sequence
with supremum $\chi_V$.  Thus
$$\nu(V)=\int_X\chi_V\,d\nu=\lim_n \int_X f_n\,d\nu \leq \lim_n  \int_X f_n\,d\mu=\int_X \chi_V\,d\mu)=\mu(V).$$
Since $V$ was an arbitrary open lower set, $\mu\leq \nu$ by Remark \ref{R:SO3}.

(2)$\Rightarrow$(1):  Property (2) implies that  $\int_X f\,d\mu\leq \int_X f\,d\nu$ for every continuous antitone function
$f:X^{od}\to \R^+$.  By the preceding paragraph $\nu\leq \mu$ with respect to $X^{od}$, i.e., $\mu\leq \nu$ with respect
to $(X,\leq)$.
\end{proof}

\begin{definition}
A topological space equipped with a closed order is called \emph{monotone normal} if given a closed upper set $A$ and a closed
lower set $B$ such that $A\cap B=\emptyset$, there exist an open upper set $U\supseteq A$ and an open lower set $V\supseteq B$ such that $U\cap V=\emptyset$.
\end{definition}

\begin{remark}\label{R-3.10}
Assume that $(X,d)$ satisfies the property stated in Proposition \ref{P:SO5} and also its dual
version that given $x\le y$ and $y_1\in X$, there is an $x_1\in X$ such that $x_1\le y_1$ and
$d(x,x_1)\le d(y,y_1)$. Then $(X,\le)$ is monotone normal as in the above definition.
Indeed, for any closed upper set $A$ and any closed lower set $B$ with $A\cap B=\emptyset$,
one can easily verify  that the open sets
$$
U:=\{x\in X:d(x,A)<d(x,B)\}\quad\mbox{and}\quad V:=\{x\in X:d(x,A)>d(x,B)\}
$$
satisfy $U\supseteq A$, $V\supseteq B$ and $U\cap V=\emptyset$.
One deduces that $U$ is an upper set and $V$ a lower from the hypothesized property and its dual.
Also, we remark that an open cone in a Banach space as considered in Section 5 satisfies the above
two properties (see Remark \ref{R-5.4}).
\end{remark}

\begin{proposition}\label{P:SO6}
Suppose that $(X,d)$ is a metric space equipped with a closed
partial order for which the space is monotone normal. Then for
$\mu,\nu\in\Pro(X)$ the following are equivalent$:$
\begin{itemize}
\item[(1)]  $\mu\leq \nu;$
\item[(2)]  For every continuous $($bounded$)$ monotone $f:X\to\R^+$, $\int_X f\,d\mu\leq \int_X f\,d\nu$.
\end{itemize}
\end{proposition}

\begin{proof}
In light of Proposition \ref{P:SO3}  we need only show condition (2) implies condition (1).
Suppose there exists some open upper set $U$ such that $\nu(U)<\mu(U)$.  By inner regularity there exists a
compact set $K\subseteq U$ such that $\nu(U)<\mu(K)\leq \mu(U)$.  The closed upper set $A=\ua K\subseteq U$ also
satisfies  $\nu(U)<\mu(A)\leq \mu(U)$.  Since $X$ is monotone normal, a modification of the usual proof of  Urysohn's  Lemma
yields a continuous monotone function $f:X\to [0,1]$ such that $f(A)=1$ and $f(X\setminus U)=0$;
see for example \cite[Exercise VI-1.16]{GS}.  We then have
$$\mu(A)=\int_X\chi_A\,d\mu\leq \int_X f\,d\mu\leq \int_X f\,d\nu\leq \int_X \chi_U\, d\nu =\nu(U),$$
a contradiction to our choice of $A$.
\end{proof}

\section{Normal cones}
Let $E$ be a Banach space containing an open cone $C$ such that its closure $\overline C$ is
a proper cone, i.e., $\overline C\cap(-\overline C)=\{0\}$.  The cone $\overline C$ defines a
closed partial order on $E$ by $x\leq y$ if $y-x\in\overline C$.  The cone $\overline C$ is
called \emph{normal} if there is a constant $K$ such that $0\leq x\leq y$ implies
$\Vert x\Vert\leq K\Vert y\Vert$.

For $x\leq y$ in $E$, the \emph{order interval} $[x,y]$ is given by
$$[x,y]:=\{w\in E: x\leq w\leq y\}=(x+\overline C)\cap (y-\overline C).$$
Note that $(x+C)\cap(y-C)$ is an open subset contained in $[x,y]$.  A subset $B$ is
\emph{order convex} if $[x,y]\subseteq B$, wherever $x,y\in B$ and $x\leq y$. An alternative
formulation of normality postulates the existence of a basis of order convex neighborhoods at
$0$ and hence by translation at all points (see Section 19.1 of \cite{De}); here neighborhood
of $x$ means a subset containing $x$ in its interior.

\begin{proposition}\label{P:NC1}
Let $E$ be a separable Banach space with an open cone $C$ such that $\overline C$ is normal.
Then restricted to $C$, the $\sigma$-algebra generated by all its open upper sets is the Borel
algebra of $C$.
\end{proposition}

\begin{proof}
Let $\A$ denote the $\sigma$-algebra of subsets of $C$ generated by the collection of all open
upper sets contained in $C$.  Fix some point $u\in C$. Let $x\in C$.  Set $r_n=1/n$ for
$n\in\N$. Then for each $n$, $x\in (x-r_nu)+C$, an open upper set,
and $\ua x=\bigcap_n[(x-r_nu)+C]$. In fact, for any $y$ in the intersection
$y-(x-r_nu)=(y-x)-r_nu\in C$ and hence the limit $y-x$ is in $\overline C$, i.e.,
$y\in x+\overline C=\ua x$. The converse inclusion is obvious since
$r_nu+\overline C\subseteq C$ so that $x+\overline C\subseteq(x-r_nu)+C$. Thus $\ua x$ is a
countable intersection of open upper sets, hence in $\A$.

Since $\da x$ is closed in $E$, $C\cap(E\setminus \da x)$ is an open upper set. Thus its
complement in $C$, which is $C\cap \da x$ is in $\A$.  Hence for $x\leq y$ in $C$, we note that
$[x,y]=\ua x\cap\da y=\ua x\cap C\cap \da y\in\A$.

Now let $U$ be a nonempty open subset of $C$.  Using the alternative characterization of normality, we may pick for each $x\in U$
an order convex neighborhood $N_x$ of $x$ that is contained in $U$. For some $\ve$ small enough $x-\ve u, x+\ve u\in N_x$, and
hence the order interval $[x-\ve u,x+\ve u]\subseteq N_x$.  Let $B_x:=(x-\ve u+C)\cap (x+\ve u-C)$, an open subset  contained in
$[x-\ve u,x+\ve u]$.  The collection $\{B_x:x\in C\}$ is an open cover of $U$, which by the separability of $E$ (and hence $U$) has
a countable subcover $\{B_{x_n}\}$.  The corresponding $[x_n-\ve_n u,x_n+\ve_n u]$  then also form a countable cover of $U$, and
since from the preceding paragraph each order interval is in $\A$, it follows that $U\in\A$.  Thus $\A$ contains all open sets of $C$, and hence
must be the Borel algebra.
\end{proof}
 We next recall E.~Dynkin's $\pi-\lambda$ theorem.  Let $X$ be a set.  A $\pi$-system is a collection of subsets of $X$ closed under finite
 intersection.  A $\lambda$-system is a collection with $X$ as a member that is closed under complementation and under countable unions
 of pairwise disjoint members of the system. An important observation is that a $\lambda$-system that is also a $\pi$-system is a $\sigma$-algebra.

\begin{theorem}  $($Dynkin's $\pi-\lambda$ Theorem$)$
If a $\pi$-system is contained in a $\lambda$-system, then the $\sigma$-algebra generated by
the $\pi$-system is contained in the $\lambda$-system.
\end{theorem}

The stochastic order on $\Pro(X)$ for a metric space $X$ equipped with a closed order is easily seen to be reflexive and transitive, but anti-symmetry is much more difficult to derive.
We now have available the tools we need to show for the open cone $C$ that the stochastic order on $\Pro(C)$ is a partial order.

\begin{theorem}\label{T:NC2}
Let $E$ be a Banach space containing an open cone $C$ such that $\overline C$ is a normal cone. Then the stochastic order on $\Pro(C)$ is a partial order.
\end{theorem}

\begin{proof}  We first consider the case that $E$ is separable.  Let $\mu,\nu\in\Pro(C)$ be
such that $\mu\leq \nu$ and $\nu\leq\mu$. We consider the set $\A$ of all Borel sets $B$ such
that $\mu(B)=\nu(B)$.  By definition of the stochastic order, $U\in \A$ for each open upper set
$U$, and the collection of open upper sets is closed under finite intersection, i.e., is a
$\pi$-system.  Since $\mu$ and $\nu$ are $\sigma$-additive measures, it follows that the collection $\A$ is closed under complementation and union of pairwise disjoint countable
families, so $\A$ is a $\lambda$-system.  By Dynkin's $\pi-\lambda$ theorem the $\sigma$-algebra
generated by the open upper sets is contained in $\A$, but by Proposition \ref{P:NC1} this is the Borel algebra.  Hence $\mu=\nu$ on the Borel algebra, that is to say $\mu=\nu$.

We turn now to the general case in which $E$ may not be separable.  In this case, however,
both $S_\mu$, the support of $\mu$, and $S_\nu$, the support of $\nu$, are separable
(Proposition \ref{P:La}).  Then also the smallest closed Banach subspace $F$ containing
$S_\mu\cup S_\nu$ will be separable, and the restrictions $\mu\vert_F$,
$\nu\vert_F\in\Pro(C\cap F)$.  Since $C\cap F$ is an open cone in $F$ with closure a normal cone, by the first part
of the proof $\mu(B)=\nu(B)$ for all Borel subsets contained in $C\cap F$. Since
$S_\mu\cup S_\nu\subseteq C\cap F$, for any Borel set $B\subseteq C$,
$$\mu(B)=\mu(B\cap S_\mu)=\mu(B\cap (S_\mu\cup S_\nu))=\nu(B\cap(S_\mu\cup S_\nu))=\nu(B\cap S_\nu)=\nu(B).$$
Thus $\mu=\nu$.
\end{proof}

\begin{remark} The techniques of the proof readily extend to any open upper set of $E$, in
particular to $E$ itself. Indeed, Proposition \ref{P:NC1} and Theorem \ref{T:NC2} hold when
restricted to any open upper set in place of $C$. So the stochastic order on $\Pro(E)$ arising
from the conic order of $E$ is also a partial order.
\end{remark}

\section{The Thompson Metric}
We continue in the setting that $E$ is a Banach space and $C$ is an open cone with its closure
$\overline C$ a normal cone.
 A. C. Thompson \cite{Thomp} has proved
that $C$ is a complete metric space with respect to the
\emph{Thompson part metric} defined by
$$d_T(x,y)={\mathrm{max}}\{\log M(x/y), \log M(y/x)\}$$ where
$M(x/y):={\mathrm{inf}}\{\lambda>0: x\leq \lambda y\}=|x|_{y}$.
Furthermore, the metric topology on $C$ arising from the Thompson metric agrees with relative topology inherited from $E$.

The contractivity of addition in $C$ with respect to the Thompson metric  has been observed in various settings
and studied in some detail in \cite{LL12}. We need only the basic formulation.
\begin{lemma}\label{L:TM1}
Addition is contractive on $C$ with respect to the Thompson metric in the sense that for all $x,y,z\in C$, $d_T(x+z,y+z)\leq d_T(x,y)$.
\end{lemma}

\begin{remark}\label{R:TM2}
The fact that the Thompson metric is complete allows us to deduce from Proposition \ref{P:La} for $E$ separable that $\Pro(C)$ consists of
all Borel probability measures and for $E$ an arbitrary Banach space that $\Pro(C)$ consists of the $\tau$-additive probability measures.
\end{remark}

\begin{proposition}\label{P:TM3}
The cone $C$ equipped with the Thompson metric satisfies the property that given $x\leq y$ and
$x_1\in C$, there exists $y_1\geq x_1$ such that $d_T(y,y_1)\leq d_T(x,x_1)$.  Hence for $\mu,\nu\in \Pro(C)$,
$\mu\leq \nu$ in the stochastic order if and only if for every continuous $($bounded$)$ monotone
$f:X\to\R^+$, $\int_X f\,d\mu\leq \int_X f\,d\nu$.
\end{proposition}

\begin{proof}
Suppose $x\leq y$ and $x_1\in C$.  The contractivity of the Thompson metric (Lemma \ref{L:TM1})
implies for $y_1=x_1+(y-x)$ that
$$d_T(y,y_1)=d_T(x+(y-x),x_1+(y-x))\leq d_T(x,x_1).$$
The last assertion of the proposition now follows from Proposition \ref{P:SO5}.
\end{proof}

\begin{remark}\label{R-5.4}
Here is a second proof of Proposition \ref{P:TM3}. Assume that $x\le y$ in $C$. For every
$x_1\in C$ let $\alpha:=d_T(x,x_1)$ so that $e^{-\alpha}x\le x_1\le e^\alpha x$. Set
$y_1:=e^\alpha y$; then $y_1\ge e^\alpha x\ge x_1$ and $y\le y_1=e^\alpha y$, so
$d_T(y,y_1)\le\alpha=d_T(x,x_1)$. Similarly one can show the dual version mentioned in
Remark \ref{R-3.10}. For every $y_1\in C$ let $\beta:=d_T(y,y_1)$ and $x_1:=e^{-\beta}x$; then
$x_1\le e^{-\beta}y\le y_1$ and $e^{-\beta}x=x_1\le x$ so that $d_T(x,x_1)\le\beta= d_T(y,y_1)$.
\end{remark}
Recall that one of the characterizations of the weak topology on any metric space, in particular on $\Pro(C)$,
 is that a net  $\mu_\alpha \to \mu$ weakly if and only if $\lim_\alpha \int_C f\, d\mu_\alpha\to \int_C f\,d\mu$ for all
continuous bounded functions into $\R$ (or $\R^+)$; see \cite{Bi}.
\begin{proposition}\label{P:TM4}
 The stochastic partial order is a closed subset of $\Pro(C)\times \Pro(C)$ endowed with the product weak topology.
 \end{proposition}

 \begin{proof}
 Let $\mu_\alpha\to \mu$ and $\nu_\alpha\to\nu$ weakly in $\Pro(C)$, where $\mu_\alpha\leq \nu_\alpha$ for each $\alpha$.
 From Proposition \ref{P:TM3} for $f:C\to \R^+$ continuous bounded and monotone
 $$\int_C f\,d\mu=\lim_\alpha \int_C f\,d\mu_\alpha\leq \lim_\alpha\int_Cf\,d\nu_\alpha=\int_C f\, d\nu.$$
 Thus again from Proposition \ref{P:TM3}, $\mu\leq \nu$.
 \end{proof}

Let $(X,\mathcal{M})$ be a \emph{measure space}, a set $X$ equipped
with a $\sigma$-algebra $\mathcal{M}$, and $(Y,d)$ a metric space. A
function  $f:X\to Y$ is \emph{measurable}  if
$f^{-1}(A)\in\mathcal{M}$  whenever $A\in\mathcal{B}(Y)$. For $f$ to
be measurable, it suffices that $f^{-1}(U)\in\mathcal{M}$ for each
open subset $U$ of $Y$.  Hence  continuous functions are measurable
in the case $X$ is a metrizable space and
$\mathcal{M}=\mathcal{B}(X)$, the Borel algebra. A measurable map
$f:X\to Y$ between metric spaces  induces the \emph{push-forward} map
$f_*:\Pro(X)\to\Pro(Y)$ defined by $f_*(\mu)(B)=\mu(f^{-1}(B))$ for
$\mu\in\Pro(X)$ and $B\in\mathcal{B}(Y)$.  Note for $f$ continuous
that $\mathrm{supp}(f_*(\mu))=f(\mathrm{supp}(\mu))^-$, the closure
of the image of the support of $\mu$.

Let $(X,d)$ be a complete metric space, and for $p\in[1,\infty)$ let
$\Pro^p(X):=\{ \mu\in\Pro(X): \int_Xd(x,y)^p\,d\mu(y)<\infty\}$, the set of $\tau$-additive Borel probability measures on $X$ with finite $p$th moment (defined independently of the choice of $x\in X$).
The {\em $p$-Wasserstein metric} $d_p^W$ on $\Pro^p(X)$ is defined by
\begin{align}\label{F-5.1}
d_p^W(\mu,\nu):=\biggl[\inf_{\pi\in\Pi(\mu,\nu)}\int_{X\times X}d(x,y)^p
\,d\pi(x,y)\biggr]^{1/p},\qquad\mu,\nu\in\Pro^p(X),
\end{align}
where $\Pi(\mu,\nu)$ is the set of all couplings for $\mu,\nu$, i.e.,
$\pi\in\Pro(X\times X)$ whose marginals are $\mu$ and $\nu$.

Recall (see, e.g.,
\cite{St}) that $\Pro^p(X)$ is a complete metric space with the metric $d_p^W$, and that
the Wasserstein convergence implies weak convergence. Hence we have the following corollary
of the preceding proposition.

\begin{corollary}
The stochastic partial order is a closed subset of $\Pro^1(C)\times \Pro^1(C)$ endowed with the product Wasserstein topology $($induced by $d_1^W$$)$.
\end{corollary}

We recall the notion of a contractive barycentric map.
\begin{definition}
Let $(X,d)$ be a complete  metric space.
   A map $\beta:\mathcal{P}^1(X)\to X$ is called a
\emph{contractive barycentric map} if
\begin{itemize}
\item[(i)] $\beta(\delta_x)=x$ for all $x\in X$;
\item[(ii)] $d(\beta(\mu),\beta(\nu))\leq d_1^W(\mu,\nu)$ for all $\mu,\nu\in\mathcal{P}^1(X)$.
\end{itemize}
For  a closed partial order $\leq $ on $X,$
$\beta:\mathcal{P}^1(X)\to X$ is said to be \emph{monotonic} if
$\beta(\mu)\leq \beta(\nu),$ whenever $\mu\leq \nu.$
\end{definition}

A complete  partially ordered metric space equipped with a monotonic contractive
barycenter has become an important object of study in recent years.

We consider the semigroup of mappings on $\psi:C\to C$ satisfying
$x\leq \psi(x)$ for all $x\in C.$ For instance, every translation
$\tau_{a}(x)=a+x$, $a\in \overline C$, satisfies this condition and also is
non-expansive for the Thompson metric.

\begin{corollary}\label{L-4}
Let $\psi:C\to C$ be a Lipschitzian map with respect to  the
Thompson metric $d_T$ such that  $x\le\psi(x)$ for all $x\in C$.
Then for every $\mu\in {\mathcal P}^{1}(C)$, we have
$\psi_*\mu\in{\mathcal P}^{1}(C)$ and $\mu\le\psi_*\mu$. If further,
$\beta:{\mathcal P}^{1}(C)\to C$ is a monotonic barycentric map,
then $\beta(\mu)\leq \beta(\psi_*\mu)$ for any $\mu\in {\mathcal
P}^1(C).$
\end{corollary}

\begin{proof}
Let $\mu=\frac{1}{n}\sum_{j=1}^{n}\delta_{x_{j}}$ be a finitely
supported uniform measure on $C$.  From $x_{j}\leq \psi(x_{j})$ for
all $j$,   we have
$$\mu=\frac{1}{n}\sum_{j=1}^{n}\delta_{x_{j}}\leq
\frac{1}{n}\sum_{j=1}^{n}\delta_{\psi(x_{j})}=\psi_*\mu.$$

Now for $\mu \in{\mathcal P}^1({\Bbb P}),$ pick  a sequence
$\mu_{n}$ of finitely supported uniform measures converging to $\mu$
from below for the Wasserstein metric associated to  the Thompson
metric (\cite[Theorem 4.7]{La17}). Then
$$\mu\leq \mu_{n}\leq \psi_*\mu_{n}\to \psi_*\mu$$
as $n\to \infty$ and hence $\mu\leq \psi_*\mu,$ by the previous
corollary.
\end{proof}

\section{Order-completeness}

In this section we always assume that the Banach space $E$ is {\em finite-dimensional} (hence
separable) and, as in Section 4, $C$ is an open cone in $E$ whose closure $\overline C$ is a
proper cone. Note (see Section 19.1 of \cite{De}) that the finite dimensionality assumption
automatically implies that $\overline C$ is a normal cone. We consider $C$ as a complete metric
space equipped with the Thompson part metric $d_T$ and the $p$-Wasserstein metric $d_p^W$ on
$\Pro^p(C)$, $1\le p<\infty$, given in \eqref{F-5.1} with $d=d_T$.

The next elementary lemma is given just for completeness.

\begin{lemma}\label{L-6.1}
\begin{itemize}
\item[(1)] For each $x,y\in C$, the order interval $[x,y]=(x+\overline C)\cap(y-\overline C)$
is a compact subset of $C$.
\item[(2)] For any $u\in C$, $\bigcup_{k=1}^\infty[k^{-1}u,ku]=C$.
\end{itemize}
\end{lemma}

\begin{proof}
(1):
Since $x+\overline C\subset C+\overline C\subset C$, $[x,y]\subset C$. It is also clear that
$[x,y]$ is a closed subset of $E$. Since $\overline C$ is a normal cone, we see that if
$z\in[x,y]$ then $\|z\|\le K\|y\|$. Hence, $[x,y]$ is a bounded closed subset of $E$. Since $E$
is finite-dimensional, $[x,y]$ is compact in $E$ and so is in $(C,d)$.

(2):
Let $u,x\in C$. For $k\in\N$ sufficiently large, $x-k^{-1}u\in C$ and  $u-k^{-1}x\in C$ so that
$x\in(k^{-1}u+C)\cap(ku-C)$. Therefore, $x\in[k^{-1}u,ku]$, which implies the assertion.
\end{proof}

Before showing order-completeness, it is convenient to derive the compactness of order
intervals in $\Pro(C)$ as well as in $\Pro^p(C)$.

\begin{proposition}\label{P-6.2}
Let $\nu_1,\nu_2\in\Pro(C)$ with $\nu_1\le\nu_2$.
\begin{itemize}
\item[(1)] The order interval $[\nu_1,\nu_2]:=\{\mu\in\Pro(C):\nu_1\le\mu\le\nu_2\}$ is compact in the
weak topology.
\item[(2)] Let $1\le p<\infty$. If $\nu_1,\nu_2\in\Pro^p(C)$, then $[\nu_1,\nu_2]\subset\Pro^p(C)$ and
it is compact in the $d_p^W$-topology.
\end{itemize}
\end{proposition}

\begin{proof}
(1):
Choose any $u\in C$. For every $\epsilon>0$ Lemma \ref{L-6.1}\,(2) implies that there exists $k\in\N$
such that $(\nu_1+\nu_2)(C\setminus[k^{-1}u,ku])<\epsilon$. We write
$C\setminus[k^{-1}u,ku]=U_k\cup V_k$, where $U_k:=\{x\in C:x\not\le ku\}$ and
$V_k:=\{x\in C:x\not\ge k^{-1}u\}$. It is clear that $U_k$ is an upper open set while $V_k$ is a lower
open set. Hence, if $\mu\in[\nu_1,\nu_2]$, then we have
\begin{align*}
\mu(C\setminus[k^{-1}u,ku])
&\le\mu(U_k)+\mu(V_k)\le\nu_2(U_k)+\nu_1(V_k) \\
&\le(\nu_1+\nu_2)(C\setminus[k^{-1}u,ku])<\epsilon
\end{align*}
(for $\mu(V_k)\le\nu_1(V_k)$, see Remark \ref{R:SO3}). By Lemma \ref{L-6.1}\,(1), this says
that $[\nu_1,\nu_2]$ is tight, and so it is relatively compact in $\Pro(C)$ in the weak topology
due to Prohorov's theorem (see \cite{Bi}). Since $[\nu_1,\nu_2]$ is closed in the weak topology
by Proposition \ref{P:TM4}, $[\nu_1,\nu_2]$ is compact in the weak topology.

(2):
Next, assume that $\nu_1,\nu_2\in\Pro^p(C)$ for $p\in[1,\infty)$. First we prove the following
``tightness" condition:
\begin{align}\label{F-1}
\lim_{R\to\infty}\sup_{\mu\in[\nu_1,\mu_2]}\int_{d(x,u)>R}d(x,u)^p\,d\mu(x)=0
\end{align}
for some $u\in C$. Choose any $u\in C$. For every $R\ge0$ set
\begin{align*}
U_R&:=\{x\in C:M(x/u)>e^R,\,M(x/u)\ge M(u/x)\}, \\
V_R&:=\{x\in C:M(u/x)>e^R,\,M(x/u)<M(u/x)\}.
\end{align*}
Then it is immediate to see that
$$
\{x\in C:d(x,u)>R\}=U_R\cup V_R\ \ \mbox{(disjoint sum)}.
$$
Hence, for any $\mu\in\Pro(C)$ we have
\begin{align}\label{F-2}
\int_{d(x,u)>R}d(x,u)^p\,d\mu(x)
=\int_C1_{U_R}(x)d(x,u)^p\,d\mu(x)+\int_C1_{V_R}(x)d(x,u)^p\,d\mu(x).
\end{align}
When $x\in U_R$ and $x\le y\in C$, since $M(x/u)\le M(y/u)$ and $M(u/x)\ge M(u/y)$, we find that
$M(y/u)\ge M(x/u)>e^R$ and $M(y/u)\ge M(u/y)$ so that $y\in U_R$. Therefore, $U_R$ is an upper Borel
set. Moreover,
$$
d(x,u)=\log M(x/u)\le\log M(y/u)=d(y,u).
$$
Hence it follows that $x\in C\mapsto1_{U_R}(x)d(x,u)^p$ is a monotone Borel function. When $x\in V_R$
and $x\ge y\in C$, since $M(x/u)\ge M(y/u)$ and $M(u/x)\le M(u/y)$, $M(u/y)\ge M(u/x)>e^R$ and
$M(y/u)<M(u/y)$ so that $y\in V_R$. Therefore, $V_R$ is a lower open set and
$$
d(x,u)=\log M(u/x)\le\log M(u/y)=d(y,u).
$$
Hence we see that $x\in C\mapsto1_{V_R}(x)d(x,u)^p$ is an antitone Borel function. If
$\mu\in[\nu_1,\nu_2]$, then by Proposition \ref{P:SO3} and Corollary \ref{C:SO4} applied to
the right-hand side of \eqref{F-2} we obtain
\begin{align*}
\int_{d(x,u)>R}d(x,u)^p\,d\mu(x)
&\le\int_C1_{U_R}(x)d(x,u)^p\,d\nu_2(x)+\int_C1_{V_R}(x)d(x,u)^p\,d\nu_1(x) \\
&\le\int_{d(x,u)>R}d(x,u)^p\,d(\nu_1+\nu_2)(x)\longrightarrow0
\end{align*}
as $R\to\infty$, since $\int_Cd(x,u)^p\,d(\nu_1+\nu_2)(x)<\infty$. Hence \eqref{F-1} has been proved,
which in particular implies that $[\nu_1,\nu_2]\subseteq\Pro^p(C)$. Moreover, from a basic fact on the
convergence in Wasserstein spaces \cite[Theorem 7.12]{Vi}, we see that $[\nu_1,\nu_2]$ is compact in
the $d_p^W$-topology. Indeed, for every sequence $\{\mu_n\}$ in $[\nu_1,\nu_2]$, from the assertion (1)
one can choose a subsequence $\{\mu_{n(m)}\}$ such that $\mu_{n(m)}\to\mu$ weakly for some
$\mu\in\Pro(C)$. Hence, it follows from \cite[Theorem 7.12]{Vi} that $\mu\in\Pro^p(C)$ and
$d_p^W(\mu_{n(m)},\mu)\to0$. Note that the limit $\mu$ is in $[\nu_1,\nu_2]$, since the
$d_p^W$-convergence implies the weak convergence. Thus, $[\nu_1,\nu_2]$ is $d_p^W$-compact.
\end{proof}

The next proposition gives the order-completeness (or a monotone convergence property) of the
stochastic order on $\Pro(C)$ in the weak topology.

\begin{proposition}\label{P-6.3}
Let $\mu_n,\nu\in\Pro(C)$ for $n\in\N$.
\begin{itemize}
\item[(1)] If $\mu_1\le\mu_2\le\dots\le\nu$, then there exists a $\mu\in\Pro(C)$ such that
$\mu_n\le\mu\le\nu$
for all $n$ and $\mu_n\to\mu$ weakly.
\item[(2)] If $\mu_1\ge\mu_2\ge\dots\ge\nu$, then there exists a $\mu\in\Pro(C)$ such that
$\mu_n\ge\mu\ge\nu$ for all $n$ and $\mu_n\to\mu$ weakly.
\end{itemize}
\end{proposition}

\begin{proof}
(1):
Since $\{\mu_n\}\subset[\mu_1,\nu]$ and Proposition \ref{P-6.2}\,(1) says that $[\mu_1,\nu]$
is compact in the weak topology, to see that $\mu_n\to\mu$ weakly for some $\mu\in\Pro(C)$,
it suffices to prove that a weak limit point of $\{\mu_n\}$ is unique. Now, let
$\mu,\mu'\in\Pro(C)$ be weak limit points of $\{\mu_n\}$, so there are subsequences
$\{\mu_{n(l)}\}$ and $\{\mu_{n(m)}\}$ such that $\mu_{n(l)}\to\mu$ and $\mu_{n(m)}\to\mu'$
weakly. Let $f:C\to[0,\infty)$ be any continuous bounded and monotone function. Since
$\int_Cf\,d\mu_n$ is increasing in $n$ by Proposition \ref{P:TM3}, we have
$$
\int_Cf\,d\mu=\lim_l\int_Cf\,d\mu_{n(l)}=\lim_m\int_Cf\,d\mu_{n(m)}=\int_Cf\,d\mu'.
$$
This implies by Proposition \ref{P:TM3} again that $\mu\le\mu'$ and $\mu'\le\mu$ so that
$\mu=\mu'$ by Theorem \ref{T:NC2}. Therefore $\mu_n\to\mu\in\Pro(C)$ weakly. Moreover, since
$\int_Cf\,d\mu_n\le\int_Cf\,d\mu\le\int_Cf\,d\nu$ for every continuous bounded and monotone
function $f\ge0$ on $C$, we have $\mu_n\le\mu\le\nu$ for all $n$.

(2):
The proof is similar to the above with a slight modification.
\end{proof}

The next proposition gives the order-completeness of the stochastic order restricted on
$\Pro^p(C)$ in the $d_p^W$-convergence.

\begin{proposition}\label{P-6.4}
Let $1\le p<\infty$ and $\mu_n,\nu\in\Pro^p(C)$ for $n\in\N$.
\begin{itemize}
\item[(1)] If $\mu_1\le\mu_2\le\dots\le\nu$, then there exists a $\mu\in\Pro^p(C)$ such that
$\mu_n\le\mu\le\nu$ for all $n$ and $d_p^W(\mu_n,\mu)\to0$.
\item[(2)] If $\mu_1\ge\mu_2\ge\dots\ge\nu$, then there exists a $\mu\in\Pro^p(C)$ such that
$\mu_n\ge\mu\ge\nu$ for all $n$ and $d_p^W(\mu_n,\mu)\to0$.
\end{itemize}
\end{proposition}

\begin{proof}
For both assertions (1) and (2), by Proposition \ref{P-6.2}\,(2) it suffices to prove that a
$d_p^W$-limit point of $\{\mu_n\}$ is unique. Since the $d_p^W$-convergence implies the weak
convergence, this is immediate from the proof of Proposition \ref{P-6.3}.
\end{proof}

\begin{corollary}\label{C-6.5}
Let $\mu,\mu_n\in\Pro(C)$, $n\in\N$. Then $\mu_n$ weakly converges to $\mu$ increasingly
$($resp.\ decreasingly$)$ in the stochastic order if and only if $\int_Cf\,d\mu_n$
increases $($resp.\ decreases$)$ to $\int_Cf\,d\mu$ for every continuous bounded and monotone
$f:C\to\R^+$. Moreover, if $\mu,\mu_n\in\Pro^p(C)$ where $1\le p<\infty$, then the above
conditions are also equivalent to $\mu_n$ converges to $\mu$ in the metric $d_p^W$ increasingly
$($resp.\ decreasingly$)$ in the stochastic order.
\end{corollary}

\begin{proof}
Assume that for any $f:C\to\R^+$ as stated above, $\int_Cf\,d\mu_n$ increases (resp.\ decreases)
to $\int_Cf\,d\mu$. Then by Proposition \ref{P:TM3}, $\mu_1\le\mu_2\le\dots\le\mu$
($\mu_1\ge\mu_2\ge\dots\ge\mu$). By Proposition \ref{P-6.3} there exists a $\mu_0\in\Pro(C)$
such that $\mu_n\to\mu_0$ weakly. By assumption, $\int_Cf\,d\mu=\int_Cf\,d\mu_0$ for any $f$
as above, which implies that $\mu=\mu_0$ by Theorem \ref{T:NC2} and Proposition \ref{P:TM3}.
Hence $\mu_n\to\mu$ weakly. Since the converse implication is obvious, the first assertion has
been shown. The second follows from Proposition \ref{P-6.4}.
\end{proof}

\begin{remark}\label{R-6.6}\rm
It is straightforward to see that $x\mapsto\delta_x$ is a homeomorphism from $(C,d_T)$ into
$\Pro(C)$ with the weak topology and also an isometry from $(C,d_T)$ into $(\Pro^1(C),d_1^W)$.
Hence each conclusion of (1) and (2) of Proposition \ref{P-6.2} implies that the interval
$[x_1,x_2]$ in $C$ is compact for any $x_1,x_2\in C$ with $x_1\le x_2$. Since
$(2^{-1}u+C)\cap(2u-C)$ is a non-empty open subset of $[2^{-1}u,2u]$ for any $u\in C$, this
forces $E$ to be finite-dimensional. Thus, the finite dimensionality of $E$ is essential in
Proposition \ref{P-6.2}. But, there might be a possibility for Propositions \ref{P-6.3} and
\ref{P-6.4} to hold true beyond the finite-dimensional case.
\end{remark}

\section{AGH mean inequalities}
In this section we consider the Banach space $E=\mathcal{B}(H)$ of
bounded operators on a (general) Hilbert space $H$ with the operator
norm, and the open cone $C=\bP$  consisting of positive invertible
operators on $H$. Note that $\bP$ is a complete metric space with
the Thompson metric $d_T$. Let $\Lambda$ be the {\em Karcher
barycenter} on ${\mathcal P}^1(\bP)$; in particular, for a finitely
and uniformly supported measure $\mu={1\over
n}\sum_{j=1}^n\delta_{A_j}$,
$$\Lambda_n(A_{1},\dots,A_{n}):=\Lambda\left(\frac{1}{n}\sum_{j=1}^{n}\delta_{A_{j}}\right)$$
is the {\em Karcher} or {\em least squares mean} of $(A_1,\dots,A_n)\in\bP^n$, which is
uniquely determined by the {\em Karcher equation}
$$\sum_{j=1}^{n}\log (X^{-1/2}A_{j}X^{-1/2})=0.$$
Moreover, $\Lambda:{\mathcal P}^1(\bP)\to \bP$ is contractive
$$
d_T(\Lambda(\mu),\Lambda(\nu))\leq d_1^W(\mu,\nu),\qquad\mu,\nu\in{\mathcal P}^1(\bP).
$$
See, e.g., \cite{LL13,LL14,LL17} for the Karcher equation and Karcher (or Cartan) barycenter.

We consider the complete metric $d_n$ on the product space $\bP^n$
\begin{align}\label{F-7.4}
d_{n}((A_{1},\dots,A_{n}),
(B_{1},\dots,B_{n})):=\frac{1}{n}\sum_{j=1}^{n}d_T(A_{j},B_{j}).
\end{align}
The contraction property of the Karcher barycenter implies that the map
$$\Lambda_{n}:\bP^{n}\to \bP,\quad (A_{1},\dots,A_{n})\mapsto
\Lambda_n(A_{1},\dots,A_{n})$$
is a Lipschitz map with Lipschitz constant $1$.

The arithmetic and harmonic means
$${\mathcal A}_{n}(A_{1},\dots,A_{n})=\frac{1}{n}\sum_{j=1}^{n}A_{j},\qquad {\mathcal
H}_{n}(A_{1},\dots,A_{n})=\left[\frac{1}{n}\sum_{j=1}^{n}A_{j}^{-1}\right]^{-1}$$
are continuous from $\bP^n$ to $\bP$ and are also Lipschitz with
Lipschitz constant $1$ for the sup-metric on $\bP^n$
\begin{align}\label{E:supm}
d_n^\infty((A_1,\dots,A_n),(B_1,\dots,B_n)):=\max_{1\le j\le n}d_T(A_j,B_j).
\end{align}

\begin{definition}
For each $n\in\N$ and $\mu_{1},\dots,\mu_{n}\in\Pro(\bP)$, note that the product measure
$\mu_1\times\dots\times\mu_n$ is in $\Pro(\bP^n)$. This is easily verified since the support
of the product measure is the product of the supports of $\mu_i$'s having the measure $1$.
As seen from Proposition \ref{P:La}, note also that the push-forward of a $\tau$-additive
measure by a continuous map is $\tau$-additive. Hence one can define the following three
measures in $\Pro(\bP)$, regarded as the geometric, arithmetic and harmonic means of
$\mu_1,\dots,\mu_n$:
\begin{align}
\Lambda(\mu_{1},\dots,\mu_{n})
&:=(\Lambda_{n})_{*}(\mu_{1}\times\cdots\times\mu_{n}), \label{F-7.6}\\
{\mathcal A}(\mu_{1},\dots,\mu_{n})
&:=({\mathcal A}_{n})_{*}(\mu_{1}\times\dots\times\mu_{n}), \label{F-7.7}\\
{\mathcal H}(\mu_{1},\dots,\mu_{n})
&:=({\mathcal H}_{n})_{*}(\mu_{1}\times\dots\times\mu_{n}). \label{F-7.8}
\end{align}
\end{definition}

\begin{example}
For $\mu=\frac{1}{n}\sum_{j=1}^{n}\delta_{A_{j}}$ and $X\in\bP$,
\begin{align*}
\Lambda(\delta_{X},\mu)&=\frac{1}{n}\sum_{j=1}^{n}\delta_{X\#A_{j}},\\
{\mathcal A}(\delta_{X},\mu)&=\frac{1}{n}\sum_{j=1}^{n}\delta_{(X+A_{j})/2},\\
{\mathcal H}(\delta_{X},\mu)&=\frac{1}{n}\sum_{j=1}^{n}\delta_{2(X^{-1}+A_{j}^{-1})^{-1}}.
\end{align*}
\end{example}

\begin{proposition}\label{P-7.3}
For every $\mu_1,\dots,\mu_n\in\Pro(\bP)$,
$$
\mathcal{H}(\mu_1,\dots,\mu_n)=\bigl[\mathcal{A}(\mu_1^{-1},\dots,\mu_n^{-1})\bigr]^{-1},
$$ where $\mu^{-1}$ is the push-forward of $\mu$ by operator
inversion $A\mapsto A^{-1}.$
\end{proposition}

\begin{proof}
For every bounded continuous function $f:\bP\to\R$ we have
\begin{align*}
&\int_\bP f(A)\,d\bigl[\mathcal{H}(\mu_1,\dots,\mu_n)\bigr]^{-1}(A)\\
&\qquad=\int_\bP f(A^{-1})\,d\mathcal{H}(\mu_1,\dots,\mu_n)(A) \\
&\qquad=\int_{\bP^n}f\biggl({1\over n}\sum_{j=1}^nA_j^{-1}\biggr)
\,d(\mu_1\times\dots\times\mu_n)(A_1,\dots,A_n) \\
&\qquad=\int_{\bP^n}f\biggl({1\over n}\sum_{j=1}^nA_j\biggr)
\,d(\mu_1^{-1}\times\dots\times\mu_n^{-1})(A_1,\dots,A_n) \\
&\qquad=\int_\bP f(A)\,d\mathcal{A}(\mu_1^{-1},\dots,\mu_n^{-1})(A),
\end{align*}
which shows that $\bigl[\mathcal{H}(\mu_1,\dots,\mu_n)\bigr]^{-1}
=\mathcal{A}(\mu_1^{-1},\dots,\mu_n^{-1})$.
\end{proof}

For a complete metric space $(X,d)$, in addition to $\Pro^p(X)$ with the $p$-Wasserstein
metric $d_p^W$ in \eqref{F-5.1} for $1\le p<\infty$, we also consider the set $\Pro^\infty(X)$
of $\mu\in\Pro(X)$ whose support is a bounded set of $X$, equipped with the
$\infty$-Wasserstein metric
\begin{equation}\label{E:winf}
d_\infty^W(\mu,\nu)= \inf_{\pi\in\Pi(\mu,\nu)}
\sup\{d(x,y):(x,y)\in\mathrm{supp}(\pi)\},
\end{equation}
where $\Pi(\mu,\nu)$ is the set of all couplings for $\mu,\nu$.

\begin{proposition}
For every $p\in[1,\infty]$ and $M=\Lambda,\mathcal{A},\mathcal{H}$ in
\eqref{F-7.6}--\eqref{F-7.8}, if $\mu_1,\dots,\mu_n\in\Pro^p(\bP)$ then
$M(\mu_1,\dots,\mu_n)\in\Pro^p(\bP)$. Moreover,
$$
(\mu_1,\dots,\mu_n)\in(\Pro^p(\bP))^n\mapsto
M(\mu_1,\dots,\mu_n)\in\Pro^p(\bP)
$$
is Lipschitz continuous with respect to the Wasserstein metric $d_p^W$.
\end{proposition}

\begin{proof}
Since $\Lambda_n:\bP^n\to\bP$ is a Lipschitz map with Lipschitz constant $1$ with respect to
$d_n$ in \eqref{F-7.4}, we can use \cite[Lemma 1.3]{LL17} to see that for each $p\in[1,\infty]$
the push-forward map $(\Lambda_n)_*:\Pro^p(\bP^n)\to\Pro^p(\bP)$ is Lipschitz with Lipschitz
constant $1$ with respect to the metric $d_p^W$, where $d_p^W$ on $\Pro^p(\bP^n)$ is defined
in terms of $d_n$. Let $\mu_1,\dots,\mu_n;\nu_1,\dots,\nu_n\in\Pro^p(\bP)$. Then it is clear
that $\mu_1\times\dots\times\mu_n\in\Pro^p(\bP^n)$ and hence $\Lambda(\mu_1,\dots,\mu_n)=
(\Lambda_n)_*(\mu_1\times\dots\times\mu_n)$ is in $\Pro^p(\bP)$. To show the Lipschitz
continuity, we may prove more precisely that
\begin{align*}
d_p^W(\Lambda(\mu_1,\dots,\mu_n),\Lambda(\nu_1,\dots,\nu_n))
&\le\Biggl[{1\over n}\sum_{j=1}^n\Bigl(d_p^W(\mu_j,\nu_j)\Bigr)^p\Biggr]^{1/p}
\quad\mbox{when $1\le p<\infty$}, \\
d_\infty^W(\Lambda(\mu_1,\dots,\mu_n),\Lambda(\nu_1,\dots,\nu_n))
&\le\max_{1\le j\le n}d_\infty^W(\mu_j,\nu_j)
\hskip2.1cm\mbox{when $p=\infty$}.
\end{align*}
To prove this, let $\pi_j\in\Pi(\mu_j,\nu_j)$, $1\le j\le n$. Since
$\pi_1\times\dots\times\pi_n\in\Pi(\mu_1\times\dots\times\mu_n,\nu_1\times\dots\times\nu_n)$,
we have, for the case $1\le p<\infty$,
\begin{align*}
&d_p^W((\Lambda_n)_*(\mu_1\times\dots\times\mu_n),
(\Lambda_n)_*(\nu_1\times\dots\times\nu_n)) \\
&\quad\le d_p^W(\mu_1\times\dots\times\mu_n,\nu_1\times\dots\times\nu_n) \\
&\quad\le\biggl[\int_{\bP^n\times\bP^n}d_n^p((A_1,\dots,A_n),(B_1,\dots,B_n))
\,d(\pi_1\times\dots\times\pi_n)\biggr]^{1/p} \\
&\quad=\Biggl[\int_{\bP^n\times\bP_n}\Biggl({1\over n}\sum_{j=1}^nd_T(A_j,B_j)\Biggr)^p
\,d(\pi_1\times\dots\times\pi_n)\Biggr]^{1/p} \\
&\quad\le\Biggl[\int_{\bP^n\times\bP_n}{1\over n}\sum_{j=1}^nd_T^p(A_j,B_j)
\,d\pi_1\times\dots\times\pi_n\Biggr]^{1/p} \\
&\quad=\Biggl[{1\over n}\sum_{j=1}^n\int_{\bP\times\bP}d_T^p(A_j,B_j)
\,d\pi_j(A_j,B_j)\Biggr]^{1/p}.
\end{align*}
By taking the infima over $\pi_j$, $1\le j\le n$, in the last
expression, we have the desired $d_p^W$-inequality when $1\le
p<\infty$. The proof when $p=\infty$ is similar, so we omit the
details.

Since $\mathcal{A}_n,\mathcal{H}_n:\bP^n\to\bP$ is Lipschitz with
Lipschitz constant $1$ with respect to $d_n^\infty$ in
(\ref{E:supm}), we can use \cite[Lemma 1.3]{LL17} again with the
metric $d_p^W$ in terms of $d_n^\infty$ (in place of $d_n$ in the
above). For the Lipschitz continuity of
$\mathcal{A}(\mu_1,\dots,\mu_n)$ we have, for $1\le p<\infty$,
\begin{align*}
&d_p^W((\mathcal{A}_n)_*(\mu_1\times\dots\times\mu_n),
(\mathcal{A}_n)_*(\nu_1\times\dots\times\nu_n)) \\
&\quad\le d_p^W(\mu_1\times\dots\times\mu_n,\nu_1\times\dots\times\nu_n) \\
&\quad\le\biggl[\int_{\bP^n\times\bP^n}\max_{1\le j\le n}d_T^p(A_j,B_j)
\,d(\pi_1\times\dots\times\pi_n)\biggr]^{1/p} \\
&\quad\le\Biggl[\sum_{j=1}^n\int_{\bP\times\bP}d_T^p(A_j,B_j)\,d\pi_j(A_j,B_j)\biggr]^{1/p},
\end{align*}
which implies that
$$
d_p^W(\mathcal{A}(\mu_1,\dots,\mu_n),\mathcal{A}(\nu_1,\dots,\nu_n))
\le\Biggl[\sum_{j=1}^n\Bigl(d_p^W(\mu_j,\nu_j)\Bigr)^p\Biggr]^{1/p}.
$$
For $p=\infty$, we similarly have
$$
d_\infty^W(\mathcal{A}(\mu_1,\dots,\mu_n),\mathcal{A}(\nu_1,\dots,\nu_n))
\le\max_{1\le j\le n}d_\infty^W(\mu_j,\nu_j).
$$
The proof for $\mathcal{H}(\mu_1,\dots,\mu_n)$ is analogous, or we may use
Proposition \ref{P-7.3}.
\end{proof}

The next theorem is the AGH mean inequalities in the stochastic order for probability measures.

\begin{theorem}\label{T-7.5}
For any $\mu_{1},\dots,\mu_{n}\in\Pro(\bP)$,
$${\mathcal H}(\mu_{1},\dots,\mu_{n})\leq
\Lambda(\mu_{1},\dots,\mu_{n})\leq {\mathcal A}(\mu_{1},\dots,\mu_{n}).$$
\end{theorem}

\begin{proof} Let
$f:X\to\R^+$ be continuous and monotone. Then by the AGH mean inequalities for operators,
\begin{align*}
\int_\bP f\,d\Lambda(\mu_{1},\dots,\mu_{n})
&=\int_{\bP^n}(f\circ\Lambda_n)(A_{1},\dots,A_{n})
\,d(\mu_{1}\times\cdots\times\mu_{n})(A_{1},\dots,A_{n})\\
&\leq\int_{\bP^n} (f\circ {\mathcal A}_n)(A_{1},\dots,A_{n})
\,d(\mu_{1}\times\cdots\times\mu_{n})(A_{1},\dots,A_{n})\\
&=\int_\bP f\,d{\mathcal A}(\mu_{1},\dots,\mu_{n}),
\end{align*}
which implies by Proposition \ref{P:TM3} that
$\Lambda(\mu_{1},\dots,\mu_{n})\leq {\mathcal A}(\mu_{1},\dots,\mu_{n})$. The proof of
${\mathcal H}(\mu_{1},\dots,\mu_{n})\leq\Lambda(\mu_{1},\dots,\mu_{n})$ is similar.
\end{proof}

\begin{theorem}\label{T-7.6}
The maps $\Lambda,\mathcal{A},\mathcal{H}:(\Pro(\bP))^n\to\Pro(\bP)$ are monotonically
increasing in the sense that if $\mu_j,\nu_j\in\Pro(\bP)$ and $\mu_j\le\nu_j$ for
$1\le j\le n$, then $M(\mu_1,\dots,\mu_n)\le M(\nu_1,\dots,\nu_n)$ for
$M=\Lambda,\mathcal{A},\mathcal{H}$.
\end{theorem}

\begin{proof}
Let $f:\bP\to\R^+$ be a monotone bounded Borel function. We write
$$
\int_\bP f\,d\Lambda(\mu_1,\dots,\mu_n)=\int_\bP g(A_1)\,d\mu_1(A_1),
$$
where
$$
g(A_1):=\int_{\bP^{n-1}}(f\circ\Lambda_n)(A_1,A_2,\dots,A_n)
\,d(\mu_2\times\dots\times\mu_n)(A_2,\dots,A_n).
$$
From the monotonicity property of $\Lambda_n$, it is immediate to see that $A_1\mapsto g(A_1)$
is a monotone bounded Borel function on $\bP$. Hence by Proposition \ref{P:SO3} we have
$$
\int_\bP f\,d\Lambda(\mu_1,\dots,\mu_n)
\le\int_\bP g(A_1)\,d\nu_1(A_1)=\int_\bP f\,d\Lambda(\nu_1,\mu_2,\dots,\mu_n).
$$
This implies that $\Lambda(\mu_1,\mu_2,\dots,\mu_n)\le\Lambda(\nu_1,\mu_2,\dots,\mu_n)$.
Repeating the argument shows that $\Lambda(\nu_1,\mu_2,\dots,\mu_n)\le
\Lambda(\nu_1,\nu_2,\mu_3,\dots,\mu_n)$ and so on. Hence $\Lambda(\mu_1,\dots,\mu_n)\le
\Lambda(\nu_1,\dots,\nu_n)$ follows. The proof is similar for $\mathcal{A}$ and $\mathcal{H}$.
\end{proof}

\begin{remark}
One can apply the arguments in this section to other multivariate
operator means of $(A_1,\dots,A_n)\in\bP^n$ having the monotonicity
property. For instance, let $P_t(A_1,\dots,A_n)$ for $t\in[-1,1]$ be
the one-parameter family of multivariate power means interpolating
$\mathcal{H}_n$, $\Lambda_n$, $\mathcal{A}_n$ as
$P_{-1}=\mathcal{H}_n$, $P_0=\Lambda_n$ and $P_1=\mathcal{A}_n$. The
power mean $P_{t}(A_{1},\dots,A_{n})$ for $t\in (0,1]$ is defined by
the unique positive definite solution of
$X=\frac{1}{n}\sum_{j=1}^{n} X\#_{t}A_{j},$ where
$A\#_{t}B=A^{1/2}(A^{-1/2}BA^{-1/2})^{t}A^{1/2}$ denotes the
$t$-weighted geometric mean of $A$ and $B.$  It is monotonic and
Lipschitz
$$d_T(P_{t}(A_{1},\dots,A_{n}), P_{t}(B_1,\dots,B_n)\leq \max_{1\leq
j\leq n}d_T(A_{j},B_{j}).$$ Moreover, $P_t(A_1,\dots,A_n)$ is
monotonically increasing in $t\in[-1,1]$ and
\begin{align}\label{F-7.10}
\lim_{t\to0}P_t(A_1,\dots,A_n)=\Lambda_n(A_1,\dots,A_n).
\end{align}
For power means, see \cite{LP} for positive definite matrices
and \cite{LL13,LL14} for positive operators on an infinite-dimensional Hilbert space.
Then one has the one-parameter family of $P_t(\mu_1,\dots,\mu_n)$ for
$\mu_1,\dots,\mu_n\in\Pro(\bP)$ so that each
$P_t(\mu_1,\dots,\mu_n)$ is monotonically increasing in
$\mu_1,\dots,\mu_n$ as in Theorem \ref{T-7.6} and
$P_t(\mu_1,\dots,\mu_n)$ is monotonically increasing in $t$,
extending the AGH mean inequalities in Theorem \ref{T-7.5}. Moreover,
$$
P_s(\mu_1,\dots,\mu_n)\leq\Lambda(\mu_{1},\dots,\mu_{n})\leq P_t(\mu_1,\dots,\mu_n)
$$
for $-1\le s<0<t\leq 1$ as in Theorem \ref{T-7.5}.

Now assume that $\bP$ is the cone of positive definite matrices of some fixed dimension, and
let $\mu_j\in\Pro^1(\bP)$, $1\le j\le n$. For any continuous bounded and monotone
$f:\bP\to\R^+$ we see by \eqref{F-7.10} that
$$
\int_\bP f\,dP_t(\mu_1,\dots,\mu_n)
=\int_{\bP^n}(f\circ P_t)(A_1,\dots,A_n)\,d(\mu_1\times\dots\times\mu_n)(A_1,\dots,A_n)
$$
increases as $t\nearrow0$ and decreases as $t\searrow0$ to
$$
\int_{\bP^n}(f\circ\Lambda_n)(A_1,\dots,A_n)\,d(\mu_1\times\dots\times\mu_n)(A_1,\dots,A_n)
=\int_\bP f\,d\Lambda(\mu_1,\dots,\mu_n).
$$
Hence by Corollary \ref{C-6.5},
$$
\lim_{t\to0}d_1^W(P_t(\mu_1,\dots,\mu_n),\Lambda(\mu_1,\dots,\mu_n))=0.
$$
It would be interesting to know whether this convergence holds true in the infinite-dimensional
case as well.
\end{remark}

\begin{remark}
Several issues arise related to $\Lambda(\mu_1,\dots,\mu_n)$. For
example, it is interesting to consider existence and uniqueness for
the least squares mean on $\Pro^1(\bP)$;
$$\underset{\mu \in \Pro^1(\bP)}{\argmin}\sum_{j=1}^{n}d_1^W(\mu, \mu_{j})^2$$
and a connection with the probability measure
$\Lambda(\mu_{1},\dots,\mu_{n}).$  Moreover, the probability Borel
measure equation
$$x={\mathcal A}(x\#_{t}\mu_{1},\dots,x\#_{t}\mu_{n}), \ \ \ \mu_{j}\in {\mathcal P}_{cp}({\Bbb P}),\ t\in (0,1],$$
where $\mu\#_{t}\nu=f_{*}(\mu\times\nu)$ is the
push-forward by the $t$-weighted geometric mean map
$f(A,B)=A\#_{t}B,$ seems to have a unique solution in ${\mathcal
P}_{cp}({\Bbb P}),$ the set of probability measures with compact
support.
\end{remark}

\section{Acknowledgements}

The work of F.~Hiai was supported in part by
Grant-in-Aid for Scientific Research (C)17K05266.
The work of Y. Lim was supported by the
National Research Foundation of Korea (NRF) grant funded by the
Korea government (MEST) No.NRF-2015R1A3A2031159.

\end{document}